\documentclass[11pt,reqno]{amsart}
\usepackage{hyperref}
\usepackage{amssymb,amsmath,amsfonts,latexsym}
\usepackage{bm}
\usepackage{mathtools}
\usepackage{enumerate, ragged2e}
\usepackage{array,graphics,color}
\newcommand\myeq{\stackrel{\mathclap{\normalfont\mbox{def}}}{=}}
\allowdisplaybreaks

\numberwithin{equation}{section}
\newtheorem{dfn}{Definition}[section]
\newtheorem{thm}[dfn]{Theorem}
\newtheorem{lma}[dfn]{Lemma}
\newtheorem{hypo}[dfn]{Hypothesis}
\newtheorem{ppsn}[dfn]{Proposition}
\newtheorem{crlre}[dfn]{Corollary}
\newtheorem{xmpl}[dfn]{Example}
\newtheorem{rmrk}[dfn]{Remark}
\newtheorem{prob}{Problem}

\DeclarePairedDelimiterX{\norm}[1]{\lVert}{\rVert}{#1}
\DeclarePairedDelimiterX{\bnorm}[1]{\big\lVert}{\big\rVert}{#1}
\DeclarePairedDelimiterX{\Bnorm}[1]{\Big\lVert}{\Big\rVert}{#1}

\newcommand{\T}{\mathbb{T}}
\newcommand{\N}{\mathbb{N}}
\newcommand{\D}{\mathbb{D}}
\newcommand{\C}{\mathbb{C}}		
\newcommand{\fcl}{\mathcal{F}}
\newcommand{\mcl}{\mathcal{M}}
\newcommand{\wcl}{\mathcal{W}}
\newcommand{\vcl}{\mathcal{V}}

\newcommand{\hile}{\mathcal{E}}

\newcommand{\kcl}{\mathcal{K}}

\newcommand{\hdcm}{{H}^2_{\mathbb{C}^m}(\mathbb{D})}
\newcommand{\hdcr}{{H}^2_{\mathbb{C}^r}(\mathbb{D})}
\newcommand{\hdcc}{{H}^2_{\mathbb{C}}(\mathbb{D})}

\newcommand{\hdcrp}{H^2(\mathbb{D},\mathbb{C}^{r\prime})}

\newcommand{\smdt}{\mathbb{E}_{H_U}(s)}
\newcommand{\kmdt}{\mathbb{E}_{K_U}(s)}
\newcommand{\bmoa}{\text{BMOA}}
\newcommand{\bmo}{\text{BMO}}
\newcommand{\smdts}{\mathbb{E}_{H_u}(s)}
\newcommand{\kmdts}{\mathbb{E}_{K_u}(s)}

\begin{document}

	\title[Schmidt subspaces of block Hankel operators]{ Schmidt subspaces of block Hankel operators}

\author[Chattopadhyay] {Arup Chattopadhyay}
\address{Department of Mathematics, Indian Institute of Technology Guwahati, Guwahati, 781039, India}
\email{arupchatt@iitg.ac.in, 2003arupchattopadhyay@gmail.com}

\author[Das]{Soma Das}
\address{Theoretical Statistics and Mathematics Unit, Indian Statistical Institute, Bangalore Centre, Bengaluru, 560059, India}
\email{dsoma994@gmail.com}

\author[Pradhan]{Chandan Pradhan}
\address{Department of Mathematics, Indian Institute of Science Bangalore,  Bengaluru, 560012, India}
\email{chandanp@iisc.ac.in, chandan.pradhan2108@gmail.com}

	

	\begin{abstract} In scalar-valued Hardy space, the class of Schmidt subspaces for a bounded Hankel operator are closely related to nearly $S^*$-invariant subspaces, as described by G\'{e}rard and Pushnitski. In this article, we prove that these subspaces in the context of vector-valued Hardy spaces are nearly $S^*$-invariant with finite defect in general. As a consequence, we obtain a short proof of the characterization results concerning the Schmidt subspaces in scalar-valued Hardy space in an alternative way. Thus, our work complements the work of G\'{e}rard and Pushnitski regarding the structure of Schmidt subspaces.
	\end{abstract}
	\maketitle
	
	\section{Introduction}\label{intro}
In the theory of operators on analytic function space, one of the significant classes of operators is the Hankel operator, as they have connections with many branches of mathematics, for example, function theory, harmonic analysis, approximation theory, moment problems, spectral theory, orthogonal polynomials, stationary Gaussian processes, etc. Hankel operators have different realizations because such a variety of realizations is essential in application. Depending on the need of the problem under consideration, one can choose a suitable realization concretely. For example, if $\Gamma =\{\gamma_{j+k}\}_{j,k=0}^\infty$ is a bounded Hankel matrix on $\ell^2(\mathbb{Z}_+)$ then one can consider $\Gamma$ as a bounded linear operator on the classical Hardy space $H^2(\T)$ using natural identification between $\ell^2(\mathbb{Z}_+)$ and $H^2(\T)$, noting that this is a linear realization. It is worth mentioning that Peller's book \cite{Peller} is a well-known and accepted reference to the classical theory of Hankel operators and their various applications. We denote the unit circle by $\T := \{z\in \C:|{z}|=1\}$ and the open unit disk by $\D :=\{z\in \C : |z|<1\}$. Let $$\hdcc := \{f=\sum_{n=0}^{\infty}a_nz^n :\sum_{n=0}^{\infty}|a_n|^2<\infty \}$$ denotes the Hardy space of analytic functions on $\D$ equipped with the norm $\norm{f}:=\left(\sum\limits_{n=0}^{\infty}|a_n|^2\right)^{1/2}$ and $L^2(\T)$ denotes the Hilbert space of square-integrable functions with respect to the normalized Lebesgue measure on $\T$. It is well-known that one can identify the closed  subspace $H^2(\T)$ of $L^2(\T)$  consisting of all functions having vanishing negative Fourier coefficients with $\hdcc$ via the radial limit. In short, we will use $H^2$ to denote the Hardy space. We refer \cite{Nikolski} for more on the Hardy space theory.

Let $P:L^2(\T) \to H^2(\T)$ be the orthogonal projection (the $Szeg\ddot{o}$ projection), then corresponding to a $u\in \bmoa(\T)$ (see \eqref{bmoc}), the anti-linear Hankel operator $H_{u}$ is defined by
\begin{equation}\label{def1}
	H_{u}(f)=P(u\bar{f}),~~ f\in H^2(\T).
\end{equation}

The symbol $u\in \text{BMOA}(\T)$ ensures that the Hankel operator $H_{u}$ is bounded, which follows from the Nehari-Fefferman theorem \cite[ Section 1.1]{Peller}. Such Hankel operator $H_{u}$ is the anti-linear realization of the Hankel matrix $$\{\hat{u}(n+m)\}_{n,m\geq0},$$
where $\hat{u}(.)$ are the Fourier coefficients of ${u}$. 
It is easy to check that the kernel of a  Hankel operator $\Gamma=H_u$ is a shift invariant subspace (see, e.g., \cite[Section 1.1]{Peller}), and hence, due to the Beurling's theorem \cite{Beurling}, it is of the form $\theta H^2$, for some inner function $\theta$. In \cite{GaPu20}, G\'{e}rard and Pushnitski raised the following question in scalar-valued Hardy space.
\begin{prob}\label{p1}
	How can one characterize eigenspaces $\ker(\Gamma ^* \Gamma - s^2I), s>0$, as a class of subspaces in the Hardy space $\hdcc$ ? 
\end{prob}
In \cite{GaPu20}, authors solved the above-mentioned problem by showing that every such subspace (known as the \textit{Schmidt} subspaces for the Hankel operator) can be identified with a subspace of the form $p\kcl_{z\theta}$, where $\theta$ is an inner function, $\kcl_{z\theta}$ is a model space, and $p$ is an isometric multiplier (will be defined later in Section \ref{prno}) on $\kcl_{z\theta}$. In addition to that, they also provide a simple formula for the action of $\Gamma$ on $\ker(\Gamma ^* \Gamma - s^2I)$ explicitly, which is completely determined by $s,p$ and $\theta$. 

It is important to note that the action of $\Gamma$ plays an important role in the description of all $s$-{\it Schmidt} pairs. Recall that \cite{GaPu20}, for a singular value $s$ of $\Gamma$, a pair $\{\xi, \eta\} \in {\ell}^2$ is called a {\it Schmidt} pair or (more precisely, an $s$-{\it Schmidt} pair) of $\Gamma$, if it satisfies $$\Gamma \xi = s\eta , \hspace{0.5cm} {\Gamma}^*\eta = s\xi .$$ The space $\ker(\Gamma ^* \Gamma - s^2I)$ is called \textit{Schmidt} subspace of $\Gamma$ and note that $s$-{\it Schmidt} pairs form a linear subspace of dimension $dim \ker(\Gamma ^* \Gamma - s^2I) \leq \infty$. Therefore the problem of description of all $s$-{\it Schmidt} pairs of $\Gamma$ is equivalent to the problem of the description of the action of $\Gamma$ on $\ker(\Gamma ^* \Gamma - s^2I)$. Also there is another advantage of this action. Suppose $\mathcal{C}$ is an anti-linear map on $\ell^2$ defined as $\mathcal{C}\xi =\bar{\xi}$, then the Hankel matrix $\Gamma$ is $\mathcal{C}$-symmetric, that is $\Gamma \mathcal{C} = \mathcal{C}{\Gamma}^*$. It is easy to observe that $$\xi \in \ker(\Gamma ^* \Gamma - s^2I) \hspace{0.5cm} \Leftrightarrow \hspace{0.5cm}  \bar{\xi} \in \ker(\Gamma  \Gamma ^* - s^2I). $$ Moreover, the anti-linear map $\Gamma \mathcal{C}$ maps $\ker(\Gamma  \Gamma ^* - s^2I)$ onto itself and the map $$ s^{-1}\Gamma \mathcal{C} : \ker(\Gamma  \Gamma ^* - s^2I) \to \ker(\Gamma  \Gamma ^* - s^2I) $$ is an involution. Therefore, using the action of $\Gamma$ on $\ker(\Gamma ^* \Gamma - s^2I)$ one can easily describe the involution map $s^{-1}\Gamma \mathcal{C}$.

Recently, in 2021, G\'{e}rard and Pushnitski \cite{GaPu21} established an excellent connection between these {\it Schmidt} subspaces with nearly $S^*$-invariant subspaces and using Hitt's \cite{Hitt} beautiful characterization of nearly $S^*$- invariant subspaces, authors gave an alternative proof (indeed, a short proof of the main result of \cite{GaPu20}) concerning the characterization of such  \textit{Schmidt} subspaces of Hankel operator.

It is important to note that there are various pieces of literature (see, e.g., \cite{paperB},\cite{GGAst}\cite{Treil}) related to the study of the spectrum of $\Gamma ^* \Gamma$  and self-adjoint  Hankel operator (see, e.g., \cite{GGEsterle},\cite{MPT}).

In this article, we extend the study of Schmidt subspaces associated with a scalar-valued Hankel operator to a matrix-valued Hankel operator. It has been noted that sometimes matrix-valued analysis is much more complicated than scalar-valued analysis due to the non-commutative property of the matrix. On the other hand, sometimes, the matrix-valued technique solves long-standing open problems for scalar cases. For example, using some new matrix-valued technique and a deep analysis of nearly $S^*$-invariant subspaces of vector-valued Hardy space, Aleman and Vukoti\'{c} \cite{AlVu09} proved that the product of finitely many Toeplitz operators on the vector-valued Hardy space is zero if and only if at least one of the symbols of the operators is singular, which answer a long-standing open problem regarding the zero-product of finitely many Toeplitz operators in scalar-valued Hardy space. Motivated by the work of G\'{e}rard and Pushnitski in \cite{GaPu20, GaPu21}, we study the Problem \ref{p1} in the case of the vector-valued Hardy spaces. More precisely, this article gives the answer to the following question completely.

\begin{prob}\label{p2}
	For a given Hankel operator $H_U$ with symmetric symbol $U\in\bmoa(\T,\C^m)$, how can one characterize eigenspaces $\ker(H_U ^* H_U - s^2I), s>0$, as a class of subspaces in the Hardy space $\hdcm$? 
\end{prob}

Moreover, we also obtain the following results in the sequel.
\begin{itemize}
	\item  In some special cases we show that Schmidt subspaces $\ker(H_U ^* H_U - s^2I), s>0$ become nearly $S^*$-invariant. 
	\item  Furthermore, we calculate precisely the action of the Hankel operator for the above mentioned cases on $\ker(H_U ^* H_U - s^2I), s>0$.
\end{itemize}
The novelty of our work lies in the fact that we obtain again a short proof of the characterization results concerning the structure of  Schmidt subspaces in scalar-valued Hardy space (see Theorem~\ref{thmscalar}) obtained by G\'{e}rard and Pushnitski as a consequence of our main results (see Theorem~\ref{nearly_defect_theorem} and \ref{main_structure_theorem})), in an alternative way compared to \cite{GaPu20, GaPu21}. 

Let us now describe about the methodology and the new tools used in our work. Because of the non-commutativity property of the symbol of the Hankel operator on $\hdcm$, several issues arise, and the classical scalar-valued methods fail to resolve them.
\begin{justify}
	\begin{itemize}
	\item[(i)]\textit{Nearly $S^*$-invariant subspaces with finite defect}: It is noted that non-trivial Schmidt subspaces $\smdt$ of $H_U$ are not always nearly invariant subspaces in the vector-valued Hardy space. So we introduce the notion of nearly $S^*$- invariant subspaces with a finite defect in $\hdcm$ (see \cite{CDP,OR}), which is one of the new ingredients in studying such spaces. We have shown that non-trivial Schmidt subspaces $\smdt$ of $H_U$ are nearly $S^*$-invariant subspaces with defect less or equal to $m$ (see Theorem \ref{nearly_defect_theorem}).
	\item[(ii)]\textit{CDP $\&$ O'Loughlin structural result for nearly $S^*$-invariant subspaces:} The second new tool is the structural theorem of nearly $S^*$-invariant subspaces in vector-valued Hardy space due to Chattopadhyay-Das-Pradh
	-an (CDP) \cite[Theorem 3.5]{CDP} and O'Loughlin \cite[Theorem 3.4]{OR}. Using these characterization results, we give a complete characterization of Schmidt subspaces in a vector-valued setting (see Theorem \ref{main_structure_theorem}).
\end{itemize}
\end{justify}

All unexplained notations used in this section are introduced and explained in the next section ({\it i.e. Section \ref{prno}}).

We end the introduction by briefly mentioning the organization of the paper.  In Section \ref{prno}, we describe all necessary notations, definitions and results related to \textit{Schmidt} subspaces of Hankel operators. Section \ref{structure} deals with the study of the complete structure of the \textit{Schmidt} subspaces in general. In Section \ref{action}, we describe the action of a specific class of Hankel operators on their \textit{Schmidt} subspaces and conclude the section with a natural question.

\section{Preliminaries and Notations}\label{prno}
We begin this section by recalling the definition of vector-valued Hardy space over the unit disk $\D$. Let $(\hile,\,\|\cdot\|_\hile)$ be a complex separable Hilbert space, then the $\hile$-valued Hardy space over the unit disk $\mathbb{D}$ is denoted by $H^2_{\hile}(\D)$ and defined by
\begin{align*}
	H^2_{\hile}(\D):=\Big\{F(z)=\sum_{n\geq 0} h_nz^n:~\|F\|^2= \sum_{n\geq 0}~\norm{h_n}_{\hile}^2<\infty,~h_n\in\hile, z\in\D\Big\}. 
\end{align*} 
Let $(\mathcal{X},\|\cdot\|_\mathcal{X})$ be a finite dimensional normed linear space. The $\mathcal{X}$-valued \\$L^p(\T, \mathcal{X})$- spaces are defined to be
\[L^p(\T, \mathcal{X}):=\{f:\T\to\mathcal{X} ~\text{measurable}~|~ \|f\|_p^p:=\int_\T \|f\|_{\mathcal{X}}^p\, dm<\infty\},\] where $dm$ is the normalized arc-length measure on $\T$. For an integer $k$ and an $L^1(\T,\mathcal{X})$-function $f$, the $k$-th order Fourier coefficient $\hat{f}(k)$ of $f$ is given by \[\hat{f}(k)=\dfrac{1}{2\pi}\int_{0}^{2\pi}f(e^{i\theta})e^{-ik\theta}\,d\theta,\]
where the above integration is interpreted by the Bochner integral sense.  For any $f\in L^1(\T,\mathcal{X})$, we define $\mathbb{P}_{\mathcal{X}}^+(f)=\sum\limits_{k=0}^{\infty}\hat{f}(k)z^k$.

Regarding this work, we will only consider $\mathcal{E}=\C^m$ for some fixed natural number $m$. We can also view $\hdcm$ as the direct sum of $m$-copies of $H^2_{\mathbb{C}}(\D)$ or as a tensor product of two Hilbert spaces $H^2_{\mathbb{C}}(\D)$ and $\C^m$, that is, $$H^2_{\hdcm}(\D)  \equiv   \underbrace{H^2_{\mathbb{C}}(\D)\oplus \cdots\oplus H^2_{\mathbb{C}}(\D)}_{m}\equiv H^2_{\mathbb{C}}(\D)\otimes \C^m.$$

Therefore $\hdcm$ can be embedded isometrically as a closed subspace of
$L^2(\mathbb{T},\C^m)$ by identifying $\hdcm$ through the non-tangential boundary limits of the $\hdcm$ functions. Let $P_m$ denotes the orthogonal projection of $L^2(\T,\C^m)$ onto $\hdcm$.  Let $\mathbb{P}_m=\mathbb{P}_{\C^m}^{+}$. Note that on $L^2(\T,\C^m)$, $\mathbb{P}_m=P_m$. Let $r,m\in\N$, and $\mathcal{L}(\C^r,\C^m)$ denote the space of all linear operators from $\C^r$ to $\C^m$. We denote the space of all holomorphic matrix valued functions by $\text{Hol}(\D,\mathcal{L}(\C^r,\C^m) )$, and by $H^\infty (\D, \mathcal{L}(\C^{r},\C^m))$ as the subspace of $\text{Hol}(\D,\mathcal{L}(\C^r,\C^m) )$, consisting of all bounded analytic functions. A function $\Theta\in H^\infty (\D, \mathcal{L}(\C^{r},\C^m)) $ is said to be an inner multiplier (inner function) if $\Theta$ is an isometry almost everywhere on the circle $\T$. Corresponding to an inner multiplier $\Theta\in H^\infty (\D, \mathcal{L}(\C^{r},\C^m))$, the model space denoted by $\kcl_\Theta
$, and is defined as
$$\kcl_\Theta := \hdcm \ominus \Theta \hdcr.$$
By an isometric multiplier on the model space $\kcl_\Theta= \hdcm \ominus \Theta \hdcr,$ we mean an analytic function $F\in \text{Hol}(\D,\mathcal{L}(\C^m,\C^n) )$ for some $n\in\N$, such that $ FG \in H^2_{\C^n}(\D)$ for every $G\in \kcl_\Theta$ and $$\norm{FG} =\norm{G} .$$
For more on isometric multipliers, we refer \cite{Crofoot}. Recall that in scalar valued case, the $\bmoa$  is the space of all analytic $\bmo$ (functions having bounded mean oscillation) functions in $\D$, in other words $\bmoa= \bmo \cap H^2$. It follows from \cite[Theorem A2.7.]{Peller} that 
\begin{align}\label{bmoc}
	\bmoa(\T):& \myeq ~~\mathbb{P}_{\C}^{+}\bmo(\T) =\mathbb{P}_{\C}^{+}L^\infty (\T) = BMO(\T)\cap H^2(\T).
\end{align}
For more details related to $\bmo$ and $\bmoa$, we refer \cite[728p-731p]{Peller}.

Let $m\in\N$, as a generalization of scalar valued $\bmoa(\T)$, we consider $\bmoa(\T,\mathcal{L}(\C^m))$ as the space of $\mathcal{L}(\C^m)$ valued $\bmoa$-functions on $\T$, which is defined to be
\begin{equation}\label{bmocm}
	\bmoa(\T,\mathcal{L}(\C^m)):\,\myeq	~~\mathbb{P}_{\mathcal{L}(\C^m)}^{+}\,\mathcal{L}^{\infty}(\T,\mathcal{L}(\C^m)).
\end{equation}   
We use the matricial notation of $U\in\bmoa(\T,\mathcal{L}(\C^m))$ as $U=[u_{ij}]_{m\times m}$
where, each $u_{ij}\in \bmoa(\T)$. Now corresponding to symbol $U\in \bmoa(\T,\mathcal{L}(\C^m))$, we define the matrix-valued Hankel operator or block Hankel operator $H_U$ on $\hdcm$ as follows:\\
For each $F=(f_1,f_2,\ldots ,f_m)\in\hdcm$,
\begin{align}\label{def2}
	H_{U}(F) := P_m(U \bar{F}), \text{ where } \bar{F}=(\bar{f_1},\bar{f_2},\ldots ,\bar{f_m}).
\end{align}
So, the above definition \eqref{def2} of $H_U$ implies that $H_U = [H_{u_{ij}}]_{m\times m}$, where $H_{u_{ij}}$ are Hankel operators on $\hdcc$.  Therefore $H_U$ is a bounded anti linear operator on $\hdcm$. Let $U=[u_{ij}]_{m\times m}\in \bmoa(\T,\mathcal{L}(\C^m))$ such that $U$ is a symmetric operator, that is, $U=U^t$ where $U^t$ is the transpose of $U$, in other words $u_{ij}=u_{ji}$ for each $i,j$. Then it turns out that $H_U^2$ is a bounded, linear, and non-negative operator on $\hdcm$. Later, we will discuss about these properties in more detail (see Proposition \ref{Properties_HU2}). For such symbol $U$ and corresponding Hankel operator $H_U$, let us denote the $\textit{Schmidt}$ subspaces in vector valued Hardy spaces $\hdcm$ by $\smdt$ and define by
\begin{align}\label{smdt}
	\smdt \, := \, \ker(H_{U}^2-s^2I), \hspace{0.5cm} s>0.
\end{align}
As discussed earlier, our main goal is to characterize these subspaces in vector valued Hardy spaces. It is important to note that $\smdt$ is an invariant subspace for $H_U$.
\begin{hypo}\label{usym}
	Assume that $U=[u_{ij}]\in\bmoa(\T,\mathcal{L}(\C^m))$ satisfy $U=U^t$, that is $u_{ij}=u_{ji}$.
\end{hypo}
\begin{ppsn}\label{Properties_HU2}
	Assume Hypothesis \ref{usym}, and let $H_U$ be a bounded Hankel anti-linear operator as defined in \eqref{def2} on $\hdcm$. Then
	\begin{enumerate}[(i)]
		\item\label{prop1} $\langle H_{U}(F),G \rangle =\langle H_{U}(G), F\rangle$ for all $F,G \in \hdcm$.
		\item\label{prop2} $H_U^2$ is a self-adjoint bounded linear operator.
		\item\label{prop3} $H_US =S^*H_U$, where $S$ is the unilateral shift on $\hdcm$ and $S^*$ is the adjoint of $S$.
	\end{enumerate} 
\end{ppsn}
\begin{proof}
	~
	\begin{enumerate}[(i)]
		\item Let $F=(f_1,f_2,\ldots ,f_m)$, $G=(g_1,g_2,\ldots ,g_m)$. Then
		\begin{align*}
			&\langle H_{U}(F),G \rangle\\
			& = \langle U\bar{F},G \rangle \\
			& = \langle u_{11}\bar{f_1}+u_{12}\bar{f_2}+\cdots +u_{1m}\bar{f_m} ,g_1 \rangle + \cdots + \langle u_{m1}\bar{f_1}+u_{m2}\bar{f_2}+\cdots +u_{mm}\bar{f_m} ,g_m \rangle \\
			& = \langle u_{11}\bar{g_1}+u_{21}\bar{g_2}+\cdots +u_{m1}\bar{g_m} ,f_1 \rangle + \cdots + \langle u_{1m}\bar{g_1}+u_{2m}\bar{g_2}+\cdots +u_{mm}\bar{g_m} ,f_m \rangle \\ 
			& = \langle u_{11}\bar{g_1}+u_{12}\bar{g_2}+\cdots +u_{1m}\bar{g_m} ,f_1 \rangle + \cdots + \langle u_{m1}\bar{g_1}+u_{m2}\bar{g_2}+\cdots +u_{mm}\bar{g_m} ,f_m \rangle \\
			& = \langle U\bar{G},F \rangle	= \langle H_{U}(G), F\rangle .
		\end{align*}
		\item  The boundedness of $H_U$ and the property \eqref{prop1} together implies the property \eqref{prop2}. 
		\item For $F\in\hdcm$, 
		\begin{align*}
			H_US(F)=P_m(U\bar{z}\bar{F})=\,&P_m(\bar{z}U\bar{F})=P_m\left[\bar{z}\left(P_m+(I-P_m)\right)(U\bar{F})\right]\\
			=\,&P_m\big(\bar{z}P_m(U\bar{F})\big)+P_m\big(\bar{z}(I-P_m)(U\bar{F})\big)\\
			=\,&P_m\big(\bar{z}P_m(U\bar{F})\big)\\
			=\,&S^*H_U(F).
		\end{align*}
	\end{enumerate}
\end{proof}

\begin{rmrk}
	Note that the last property \eqref{prop3} holds true for any $U=[u_{ij}]\in\bmoa$, and it is basically characterizes all the bounded Hankel operators on $\hdcm$. Then the following identity follows from the above Proposition \ref{Properties_HU2}.
	\begin{equation*}
		S^*H_U^2S = H_U^2 - \sum_{i=1}^{m}\left \langle \cdot, U_i \right \rangle U_i,
	\end{equation*}
	where  $U_i=[u_{1i},u_{2i},\ldots,u_{mi}]^t$ for $i\in \{1,2,\ldots ,m\}$.
\end{rmrk}  
Going further, we need the help of the following auxiliary Hankel operator $K_U$ (in scalar-valued case the notion of such kind of operator mentioned in \cite{GaPu20, GaPu21}):
\begin{align}\label{ku}
	K_U :&= H_U S
	=S^*H_U =H_{S^*U}
\end{align}
on $\hdcm$, where $U$ satisfies the Hypothesis \ref{usym}.
Depending on the context, we will use suitable definition of $K_U$. 
Since we are in the vector-valued Hardy space $\hdcm$, then we have
\begin{align*}
	SS^* &= I_{\hdcm} -P_{_{\C^m}} \\
	& = I_{\hdcm} - \sum_{i=1}^{m} \langle \cdot, k_{0}\otimes e_i \rangle k_{0}\otimes e_i,
\end{align*}
where $P_{_{\C^m}}$ is the orthogonal projection from $\hdcm$ onto  $\C^m$, and $\{e_i\}_{i=1}^m$ is the standard orthonormal basis of $\C^m$, and $k_{0}$ is the reproducing kernel of $\hdcc$ at $0$. 

Recall that $U=[u_{ij}]_{m\times m}$ with $u_{ij}=u_{ji}$. For each $i\in \{1,2,\ldots ,m\}$, let us denote the ith column of $U$ by $$U_i=[u_{1i},u_{2i},\ldots,u_{mi}]^t.$$
Then it turns out,
\begin{align*}
	H_U(e_i) &= U_i=[u_{1i},u_{2i},\ldots,u_{mi}]^t .
\end{align*}
From the definition \eqref{ku} of $K_U$, it follows that $K_U$ is a bounded anti linear operator, which further imply that $K_U ^2$ is a bounded linear operator on $\hdcm$. Next we give a precise form of the operator $K_U^2$ in terms of $H_U^2$ as follows
\begin{align*}
	{K^2_U} & = (H_U S)(H_U S) = H_U SS^* H_U \\
	& = H_U \left(I_{\hdcm} - \sum_{i=1}^{m} \langle \cdot, k_{0}\otimes e_i \rangle k_{0}\otimes e_i\right) H_U \\
	& = {H^2_U} - \sum_{i=1}^{m} \left \langle \cdot, U_i \right \rangle U_i ,
\end{align*} 
where $\sum\limits_{i=1}^{m}\langle \cdot, U_i  \rangle U_i$ is a finite rank operator generated by $\{U_1,U_2,\ldots U_m\}$ and having rank at most $m$. Moreover, this form of $K_U^2$ guarantees that $K_U^2$ is self-adjoint. Like the \textit{Schmidt} subspaces of $H_U$, similarly we also define the eigenspaces related to $K_U$ as follows:
\begin{align}
	\kmdt \, := \, \ker(K_{U}^2-s^2I), ~~s>0,
\end{align}  
and these eigenspaces $\kmdt$ are basically the \textit{Schmidt} subspaces corresponding to the Hankel operator $K_U(=H_{S^*U})$. 
The next observation is crucial to characterize the \textit{Schmidt} subspaces of $H_U$.
\begin{lma}\label{ehu_eku_relation}
	Let $U$ satisfy the Hypothesis \ref{usym}. Then 
	\begin{align*}
		\smdt \cap \{U_1,U_2,\ldots ,U_m\}^\perp = \kmdt \cap \{U_1,U_2,\ldots ,U_m\}^\perp,
	\end{align*}
	where  $U_i=[u_{1i},u_{2i},\ldots,u_{mi}]^t$ for $i\in \{1,2,\ldots ,m\}$, and $$\{U_1,U_2,\ldots ,U_m\}^\perp=\{F\in\hdcm: \langle F,U_i \rangle=0 \text{ for } 1\leq i\leq m\}.$$
\end{lma}
\begin{proof}
	The proof is trivial, follows from the definition of the corresponding subspaces.
\end{proof}
Now, let us recall some useful existing definition in literature required to understand the structure of \textit{Schmidt} subspaces. 
\begin{dfn}
	A closed subspace $\mathcal{M}$ of $\hdcm$ is said to be nearly $S^*$-invariant if every element  $F\in \mathcal{M}$ with $F(0)=0$ satisfies $S^*F\in \mathcal{M}$. As a generalization, we call a closed subspace $\mcl \subset \hdcm$ to be nearly $S^*$-invariant with defect p if and only if there is an $p$-dimensional subspace $\fcl \subset \hdcm$ (which may be taken to be orthogonal to $\mcl$ ) such that if $F \in \mathcal{M},F(0)=0$ then $S^*F$ $\in \mathcal{M}\oplus \fcl$.
\end{dfn}
Nearly $S^*$-invariant subspaces were first introduced by Hitt \cite{Hitt} and further studied by various authors (see, e.g.,  \cite{Aleman,Hayashi,Sarason}). These subspaces were also studied in connection with kernels of Toeplitz operator (see, e.g., \cite{Partington,Dyakonov,Fricain,Hartmann}).
Characterization of \textit{nearly $S^*$-invariant} subspaces in the vector-valued Hardy space was obtained by  Chalendar, Chevrot, and Partington (C-C-P) \cite{CCP}, which provides a vectorial generalization of  Hitt's \cite{Hitt} characterization.  

\begin{thm} \cite[Theorem 4.4]{CCP}\label{ccp_thm}
	Let $\fcl$ be a nearly $S^*$-invariant subspace of $\hdcm$ and let $(W_1,W_2,\ldots,W_r)$ be an orthonormal basis of \hspace{.05in} $\wcl : = \fcl \ominus (\fcl~\cap z\hdcm)$. Let $F_0$ be the $m\times r$ matrix whose columns are $W_1,W_2,\ldots,W_r$. Then there exists an isometric mapping $J:\fcl \rightarrow \fcl ^\prime$ given by $F_0G\mapsto G$, where $\fcl ^\prime:=\{G\in \hdcr :\exists F\in\fcl,F=F_0G \}$. Moreover $\fcl ^\prime$ is $S^*$ invariant.
\end{thm}
As a corollary of this theorem they have the following result,
\begin{crlre}\label{ccp_corollary}Assume the notations used in the above Theorem \ref{ccp_thm}. Then there exists an inner function $\Phi\in H^\infty(\D,\mathcal{L}(\mathbb{C}^{r\prime},\mathbb{C}^r))$, which is unique up to unitary equivalence and vanishes at zero, such that $\fcl= F_0(\hdcr \ominus \Phi\hdcrp).$
\end{crlre}
Recently, in \cite[Theorem 3.5]{CDP}, the authors have provided a complete characterization of \textit{nearly $S^*$-invariant} subspaces with finite defect in the vector-valued Hardy space $\hdcm$, which extends the result of Chalendar, Gallardo, Partington \cite{CGP} from scalar valued setting to vector-valued setting. In this connection, it is worth mentioning that, O'Loughlin \cite{OR} obtained the similar structural result for these \textit{nearly $S^*$-invariant}  subspaces independently. Note that Theorem 3.5 in \cite{CDP} (see also, \cite[Theorem 3.4]{OR}) is one of the key tools to solve our problem. For more applications on the characterization results of  \textit{nearly $S^*$-invariant} subspaces with finite defect $\hdcm$, we refer to \cite{CDP2,CD}.

\section{The Structure of $\smdt$}\label{structure}
The following lemma will be useful to prove one of our main results in this section. Throughout this section, let $U$ satisfy the Hypothesis \ref{usym}, and for $i\in\{1,2,\ldots ,m \}$, let $U_i =U(e_i)$, where $\{e_i:1\leq i\leq m\}$ is the standard orthonormal basis of $\C^m$. Then we have the following lemma:
\begin{lma}\label{dim_lemma}
	Let $\vcl:=\, \kmdt \ominus \left\{\kmdt \cap \{U_1,U_2,\ldots ,U_m\}^\perp \right\}$.
	 Then $$dim(\vcl) \leq m .$$
\end{lma}
\begin{proof}
	We first consider the set $$\mcl :=\left\{ P_{\kmdt}(U_1),P_{\kmdt}(U_2),\ldots ,P_{\kmdt}(U_m)  \right\} \subseteq \kmdt.$$ Our claim is that $\mcl$ generates the subspace $\vcl$.
	
	For $f\in \kmdt \cap \{U_1,U_2,\ldots ,U_m\}^\perp$ and each $i\in \{1,2,\ldots ,m\}$, we have
	\begin{align*}
		&\left\langle P_{\kmdt}(U_i), f \right\rangle  = \left\langle U_i, f   \right\rangle =0.
	\end{align*}
	Therefore  $ P_{\kmdt}(u_i) \in \vcl $. Next, assume  $g\in \vcl\cap\mcl^\perp$. Then, for each $i\in \{1,2,\ldots ,m\}$, we have
	\begin{align*}
		\left\langle g, P_{\kmdt}(u_i)\right\rangle =0&\implies \left\langle g,u_i  \right\rangle =0\\
		&\implies g\in \kmdt \cap \{U_1,U_2,\ldots ,U_m\}^\perp\implies g=0.
	\end{align*}
	Hence, we obtain $$\operatorname{span} \left\{ P_{\kmdt}(U_1),P_{\kmdt}(U_2),\ldots ,P_{\kmdt}(U_m)  \right\} = \vcl .$$ 
	Therefore, $dim(\vcl) \leq m $. 
\end{proof}
Now we are in a position to state and prove the main theorem of this section.
\begin{thm}\label{nearly_defect_theorem}
	Assume Hypothesis \ref{usym}, and $H_U$ is the bounded Hankel operator on $\hdcm$ given by \eqref{def2}. Then, for each $s>0$, the \textit{Schmidt} subspace $\smdt$ of $H_U$ is a \textit{nearly $S^*$-invariant} subspace of $\hdcm$ with defect at most $m$.
\end{thm}
\begin{proof}Let $s>0$ be fixed.
	Let $F\in \smdt \cap z\hdcm$. Note that $F(0)=0$, so we can write $F =SG$ for some $G\in \hdcm$. Therefore, we have $H^2_{U}(F)=s^2 F$, and
	\begin{align*}
		K^2_{U}(G) & = S^* H^2_{U}S(G)
		= S^*H^2_{U}(F)
		= s^2 S^*(F) = s^2 G,
	\end{align*}
	which gives $G\in \kmdt$. That means $S^*(F)\in \kmdt$, which basically implies that
	\begin{align}\label{cont1}
		S^*(\smdt \cap z\hdcm ) \subseteq \kmdt .
	\end{align} 
	Let $\vcl:=\, \kmdt \ominus \left\{\kmdt \cap \{U_1,U_2,\ldots ,U_m\}^\perp \right\}$. Now we can decompose the subspace $\kmdt$ as a direct sum two subspaces as follows:
	\begin{align}\label{exp1}
		\kmdt & =\left(\kmdt \cap \{U_1,U_2,\ldots ,U_m\}^\perp\right) \oplus  \vcl.  
	\end{align}
	From Lemma \ref{ehu_eku_relation}, it follows that $$\smdt \cap \{U_1,U_2,\ldots ,U_m\}^\perp = \kmdt \cap \{U_1,U_2,\ldots ,U_m\}^\perp.$$ Therefore the above decomposition \eqref{exp1} can be rewritten as,
	\begin{align*}
		\kmdt & = \left(\smdt \cap \{U_1,U_2,\ldots ,U_m\}^\perp \right) \oplus \vcl .
	\end{align*}
	Hence \eqref{cont1} finally turns to
	\begin{align}\label{final}
		S^*\left(\smdt \cap z\hdcm \right) \subseteq \left( \smdt \cap \{U_1,U_2,\ldots ,U_m\}^\perp \right) \oplus  \vcl.
	\end{align}
	Therefore using the inclusion \eqref{final} and using Lemma \ref{dim_lemma} we conclude that $\smdt$ is a \textit{nearly $S^*$-invariant} subspaces with defect at most $m$.
\end{proof}
\noindent The next theorem gives the characterization of \textit{Schmidt} subspaces in $\hdcm$.
\begin{thm}\label{main_structure_theorem}
	Let $\smdt$ be the \textit{Schmidt} subspaces of the bounded Hankel operator $H_U$ with $U=U^t$. Then for each $s>0$, there exist an orthonormal set of vectors $\{E_1,E_2,\ldots ,E_p \}$ in $\hdcm$ such that $\{E_1,E_2,\ldots ,E_p \} \perp \smdt$ for some non-negative integer $p$ satisfying $p\leq m$, such that $\smdt$ is \textit{nearly $S^*$-invariant} of defect $p$ having the defect space $\fcl=\bigvee_{i=1}^p\{E_i\}$. Moreover, if $\{W_1,W_2,\ldots,W_r\}$ be an orthonormal basis of $\wcl:=\smdt\ominus(\smdt\cap z\hdcm)$ and let $F_0$ be the $m\times r$ matrix whose columns are $W_1,W_2,\ldots,W_r$. Then
	\begin{enumerate}[(i)]
		\item\label{s1} in the case when $ \smdt\cap z\hdcm\neq \smdt$,
		\begin{equation*}
			\smdt = \Big\{F : F(z)= F_0(z)K_0(z)+ \sum_{j=1}^{p} zk_j(z)E_j(z) : (K_0,k_1,\ldots,k_p)\in {\kcl}_{\Theta} \Big\},
		\end{equation*}
		where 
		$${\kcl}_{\Theta} = H^2_{\mathbb{C}^{r+p}}(\mathbb{D}) \ominus \Theta \hdcrp,$$ for some inner function $\Theta\in H^\infty(\D,\mathcal{L}(\mathbb{C}^{r\prime},\mathbb{C}^{r+p}))$ and   
		$$\norm{F}^2=\norm{K_0}^2+\sum_{j=1}^{p}\norm{k_j}^2 .$$
		\item\label{s3} In the case when $\smdt\cap z\hdcm= \smdt$, 
		\begin{equation*}
			\smdt = \Big\{F : F(z)= \sum_{j=1}^{p} zk_j(z)E_j(z) : (k_1,\ldots,k_p)\in \mathcal{K}_{\Phi} \Big\},
		\end{equation*}
		with the same notion as in \eqref{s1} except that $\mathcal{K}_{\Phi}= H^2_{\mathbb{C}^{p}}(\mathbb{D})\ominus \Phi H^2_{\mathbb{C}^{p^\prime}}(\mathbb{D})$ for some inner function $\Phi\in H^\infty(\D,\mathcal{L}(\mathbb{C}^{p^\prime},\mathbb{C}^{p}))$ and $\norm{F}^2=\sum_{j=1}^{p}\norm{k_j}^2 .$
	\end{enumerate}
\end{thm}
\begin{proof}
	Let $s>0$. By the above Theorem \ref{nearly_defect_theorem}, it follows that $\smdt$ is a \textit{nearly $S^*$-invariant} subspace of $\hdcm$, with defect at most $m$. Therefore, for each $s>0$ there exist a non-negative integer $p$ ($p\leq m$) such that $p$ is the dimension of the defect space of $\smdt$. Now by applying C-D-P theorem \cite[Theorem 3.5]{CDP} on $\smdt$, and finally using the \textit{Beurling-Lax-Halmos} characterization \cite[Theorem 3.3]{NaFo70}, we conclude the proof.
\end{proof}
Representation of \textit{Schmidt} subspaces given by Theorem \ref{main_structure_theorem} provides a vectorial generalization of recent known characterization results \cite[Theorem 1.5]{GaPu20} (\cite[Theorem 1.3]{GaPu21}) of Schmidt subspaces in scalar-valued Hardy space, as one can get back the representation of such subspaces in $\hdcc$ by using our results but in a different way. Indeed, we restate \cite[Theorem 1.5]{GaPu20} (\cite[Theorem 1.3]{GaPu21}) and provide a simple alternative short proof again by using Theorem \ref{main_structure_theorem}.
\begin{thm}\label{thmscalar}
	Let $u\in \bmoa(\T)$ and $H_u$ be the corresponding Hankel operator on $\hdcc$ given by \eqref{def1}. Then each non-trivial \textit{Schmidt} subspace $\smdts$, $s>0$, is of the form $h\kcl_{\theta}$, where $\theta$ is an inner function in $\hdcc$ and $h$ is an isometric multiplier on the model space $\kcl_\theta\subseteq\hdcc$.
\end{thm}
\begin{proof}
	Let $\smdts$ be a non-trivial \textit{Schmidt} subspace of the Hankel operator $H_u$ for some $s>0$. Then by Theorem \ref{main_structure_theorem}, $\smdts$ is a \textit{nearly $S^*$-invariant} subspaces having defect at most $1$. Now we will prove the theorem by analyzing the following two cases.
	
	\textbf{Case I:} Suppose  $~\smdts \nsubseteq z\hdcc$.	In this case our aim is to show that $\smdts$ is nearly $S^*$ invariant with defect $0$, that is just nearly invariant.  To that aim, let $\fcl$ be the defect space having dimension at most $1$. By mimicking the proof of Lemma \ref{dim_lemma}, we conclude that $\fcl = \langle P_{\kmdts}(u)\rangle$. Now for any $f\in \kmdts $, by definition we have $K_u^2(f)=s^2f$, that is $H_u^2(f)-\langle f,u \rangle u = s^2f$. Let $0\neq g\in \smdts \ominus \left(\smdts\cap z\hdcc \right)$. Then
	\begin{align*}
		s^2 \langle f,H_u(g) \rangle &= \langle s^2f,H_u(g) \rangle\\
		& = \langle H_u^2(f),H_u(g) \rangle - \langle f,u \rangle \langle u,H_u(g)\rangle \\
		& = \langle f,H_uH_u^2(g) \rangle - \langle f,u \rangle \langle H_u(1),H_u(g)\rangle \\
		& = s^2 \langle f,H_u(g) \rangle - \langle f,u \rangle \langle H_u^2(g),1\rangle \\
		& = s^2 \langle f,H_u(g) \rangle - s^2\langle f,u \rangle \langle g,1\rangle,
	\end{align*}
	which implies
	\begin{align}\label{s2}
		s^2\langle f,u \rangle \langle g,1\rangle =0.
	\end{align} Since $g(\neq 0)\in \smdts \ominus \left(\smdts\cap z\hdcc \right)$ then $\langle 1,g\rangle \neq 0$, so \eqref{s2} implies that $\langle f,u \rangle =0$. Therefore, under the assumption $~\smdts \nsubseteq z\hdcc$ we get $u \perp \kmdts$ and hence the defect space $\fcl =\{0\}$. Consequently, $\smdts$ is \textit{nearly $S^*$-invariant} and by Theorem \ref{main_structure_theorem} \eqref{s1}, $\smdts = \big\{f : f(z)= f_0(z)K_0(z) : K_0\in {\kcl} \big\}$ where $\kcl = \kcl_\phi \subset \hdcc$ for some inner function $\phi$ together with $\norm{f}^2=\norm{k_0}^2$.
	
	\textbf{Case II:} Suppose $~\smdts \subseteq z\hdcc$. In this case, it is important to note that $\smdt$  is not \textit{nearly $S^*$-invariant} but by Theorem \ref{nearly_defect_theorem}, it is \textit{nearly $S^*$-invariant} with defect $1$.  So, the defect space $\fcl \neq \{0\}$ and by Theorem \ref{main_structure_theorem} \eqref{s3}, $\smdts = \big\{f : f(z)= zk_1(z)E_1(z) : k_1\in {\kcl} \big\}$, where $\{E_1\}$ is a basis of the defect space $\fcl$ and $\kcl = \kcl_\psi \subset \hdcc$ for some inner function $\psi$ together with $\norm{f}^2=\norm{k_1}^2$.
	
	Therefore in both cases, $\smdts = h\kcl_{\theta}$, where $\theta$ is an inner function and $h$ is an isometric multiplier on $\kcl_\theta$.
\end{proof}

Motivating by the representation of $\smdts$ in scalar valued Hardy space $\hdcc$ it is natural to investigate under what circumstances the defect space $\fcl$ will be zero in vector-valued setting. In other words, we would like to ask the following question:
\begin{prob}\label{q2}
	When the \textit{Schmidt} subspaces will become nearly $S^*$-invariant in vector valued Hardy space $\hdcm$ ? 
\end{prob} 

Now it is trivial to check that if $~\smdt \subseteq z\hdcm$ then it cannot be \textit{nearly $S^*$-invariant}. So, we have to investigate under the assumption $\smdt \nsubseteq z\hdcm$. It is clear that if $\{U_1,U_2,\ldots ,U_m\}\perp \kmdt $, then the defect space $\fcl$ becomes trivial and hence $\smdt$ is \textit{nearly $S^*$-invariant} in $\hdcm$. In the following theorem we will provide another important necessary condition which makes $\smdt$ to be \textit{nearly $S^*$-invariant}.
\begin{thm}\label{nearly_inv_thm}
		Let $~\smdt \nsubseteq z\hdcm$. If $$\text{dim}(\smdt \ominus (\smdt \cap z\hdcm))=m,$$ then $\smdt$ is nearly $S^*$-invariant in $\hdcm$ and $\smdt = F_0(\hdcm \ominus \Phi\hdcr$ $)$ for some inner function $\Phi\in H^\infty(\D,\mathcal{L}(\mathbb{C}^{r},\mathbb{C}^m))$, with $\Phi (0) = 0 $ and $F_0$ is a $m\times m$ matrix whose columns are an orthonormal basis of $\smdt \ominus (\smdt \cap z\hdcm)$.
\end{thm} 
\begin{proof}
	We consider an element $F\in \smdt$ with $F(0)=0$. Then it follows from the set inclusion \eqref{final}  that  $S^*(F)\in \left(\smdt \cap \{U_1,U_2,\ldots ,U_m\}^\perp \right) \oplus  \vcl $, where 
	$$\vcl=\, \kmdt \ominus \left\{\kmdt \cap \{U_1,U_2,\ldots ,U_m\}^\perp \right\}. $$ Note that $\vcl$ is generated by $\left\{ P_{\kmdt}(U_1),P_{\kmdt}(U_2),\ldots ,P_{\kmdt}(U_m)  \right\}$.  
	
	By hypothesis, $\text{dim}(\smdt \ominus (\smdt \cap z\hdcm))=m $. Let $\{W_1,W_2,\ldots ,$ $W_m\}$ to an orthonormal basis of $\vcl$. Then $\{W_1(0),W_2(0),\ldots ,W_m(0)\}$ is linearly independent in $\C^m$. Indeed, suppose for $\alpha_1, \alpha_2,\ldots ,\alpha_m \in \C$,
	\begin{align*}
		\alpha_1 W_1(0) +\alpha_2W_2(0)+\cdots +\alpha_mW_m(0)=0.	
	\end{align*}
	Let $G= \alpha_1 W_1 +\alpha_2W_2+\cdots +\alpha_mW_m$, then $G\in \smdt \cap z\hdcm$. So, for each $i\in \{1,2,\ldots ,m\}$, we have $\langle G,W_i\rangle = 0$ which implies $\alpha_i =0$. 
	
	Take $J\in \kmdt$ then $K_U^2(J)={H^2_U}(J) - \sum\limits_{i=1}^{m} \left \langle J, U_i \right \rangle Ui =s^2 J$. Note that $H_U(\smdt) \subset \smdt$. Therefore for any $r\in \{1,2,\ldots ,m \}$ we have,
	\begin{align*}
		s^2\left\langle J,H_U(W_r)\right\rangle & = \left\langle H_{U}^2(J), H_U(W_r) \right\rangle - \sum_{i=1}^{m} \left \langle J, U_i \right \rangle \left\langle U_i, H_U(W_r) \right\rangle \\
		& = \left\langle J, H_UH_{U}^2(W_r) \right\rangle - \sum_{i=1}^{m} \left \langle J, U_i \right \rangle \left\langle H_U(e_i), H_U(W_r) \right\rangle \\
		& = s^2\left\langle J, H_U(W_r) \right\rangle - \sum_{i=1}^{m} \left \langle J, U_i \right \rangle \left\langle  H^2_U(W_r),e_i \right\rangle \\
		& = s^2\left\langle J, H_U(W_r) \right\rangle - s^2\sum_{i=1}^{m} \left \langle J, U_i \right \rangle \left\langle  W_r,e_i \right\rangle,
	\end{align*}
	where $\{e_1,e_2,\ldots ,e_m\}$ is the standard orthonormal basis of $\C^m$. So, we obtain \begin{equation}\label{inv_eq}
		\sum_{i=1}^{m} \left \langle J, U_i \right \rangle \left\langle  W_r,e_i \right\rangle = 0,\quad r\in \{1,2,\ldots ,m \}.
	\end{equation}
	Let $W$ be an $m\times m$ matrix whose columns are $W_1,W_2,\ldots ,W_m$. Therefore $W(0)= [W_1(0),W_2(0),$ $\ldots ,W_m(0)]_{m\times m}$ is an invertible matrix in $M_{m}(\C)$. Now the equations derived in \eqref{inv_eq} can be rewritten as 
	\begin{align*}
		[W(0)]{^t}_{m\times m}[\langle J,U_1 \rangle , \langle J,U_2 \rangle ,\cdots ,\langle J,U_m \rangle ]_{m\times 1}^t  = 0,
	\end{align*}
	which implies 
	\begin{align*}
		[\langle J,U_1 \rangle , \langle J,U_2 \rangle ,\cdots ,\langle J,U_m \rangle ]_{m\times 1}^t  = 0.	
	\end{align*}
	So, we have $\langle J,U_i \rangle = 0 , \forall i\in \{1,2,\ldots ,m\}$ and hence $J\perp \{U_1,U_2,\ldots ,U_m\}$ for $J\in \kmdt$. Therefore, $\kmdt \perp \{U_1,U_2,\ldots ,U_m\}$, which further implies $\vcl =\{0\}$. In other words, $\smdt$ is \textit{nearly $S^*$-invariant} in $\hdcm$ and 
	\begin{equation}\label{importantnew}
		S^*(\smdt \cap {\C^m}^\perp) \subset \smdt \cap \{U_1,U_2,\ldots ,U_m\}^\perp .
	\end{equation}
	Therefore, by Corollary \ref{ccp_corollary}, there exists an inner function $\Phi\in H^\infty(\D,\mathcal{L}(\mathbb{C}^{r},\mathbb{C}^m))$, with $\Phi (0) = 0 $ such that $\smdt = F_0(\hdcm \ominus \Phi\hdcr)$, where $F_0 = [W_1,W_2,$ $\ldots ,W_m]_{m \times m}$.  
\end{proof}
\noindent In the above Theorem \ref{nearly_inv_thm}, the assumption $\text{dim}(\smdt \ominus (\smdt \cap z\hdcm))=m$ is quite natural, since in the scalar-valued case if   $\smdts \nsubseteq z\hdcc$, then we must have $$dim(\smdts \ominus \left(\smdts\cap z\hdcc \right))=1.$$ 
In the vector-valued Hardy space, there are plenty of examples of \textit{Schmidt} subspaces satisfying  the above assumption. Indeed, in particular, if we are in $H^2_{\C^2}(\D)$, then the following examples serve our purpose.
\begin{xmpl}\label{xmpl3.6}
	Let $\phi \in \mathcal{H}^\infty (\D,\C)$  be an inner function, and consider
	\begin{enumerate}[(A)]
		\item  
		\begin{equation*}
			U=
			\begin{bmatrix}  
				\phi & 0 \\[1pt]
				0 & \phi
			\end{bmatrix}, 
			\hspace{0.5cm} \textit{then} \hspace{0.5cm}
			H_U ^2=
			\begin{bmatrix}  
				H_{\phi}^2 & 0 \\[1pt]
				0 & H_{\phi}^2
			\end{bmatrix}. 
		\end{equation*}
		Therefore the \textit{Schmidt} subspace $\smdt =\ker(H_U ^2 -s^2 I) = \ker(H_{\phi}^2 -s^2I) \oplus \ker(H_{\phi}^2 -s^2I)$.\\
		\item 
		\begin{equation*}
			U=
			\begin{bmatrix}  
				0 & \phi \\[1pt]
				\phi & 0
			\end{bmatrix}, 
			\hspace{0.5cm} \textit{then} \hspace{0.5cm}
			H_U ^2=
			\begin{bmatrix}  
				H_{\phi}^2 & 0 \\[1pt]
				0 & H_{\phi}^2
			\end{bmatrix}. 
		\end{equation*}
		In this case also $\smdt =\ker(H_U ^2 -s^2 I) = \ker(H_{\phi}^2 -s^2I) \oplus \ker(H_{\phi}^2 -s^2I)$.
	\end{enumerate}
	Note that, in both cases $\text{dim}(\smdt \ominus (\smdt \cap zH^2_{C^2}(\D)))= 2$.
\end{xmpl}

In the next section, we will provide some non-trivial examples of class of \textit{Schmidt} subspaces satisfying such assumption (see Example~\ref{lastexample}). 
We end the section with the following remark:  
\begin{rmrk}
	As mentioned in [\cite{GaPu20}, Appendix], one can have similar results regarding the linear Hankel operator in vector-valued setting as well. Suppose $J$ is the linear involution in $L^2(\T , \C^m)$ defined as $$JF(z)=(f_1(\bar{z}),f_2(\bar{z}),\ldots ,f_m(\bar{z})),  {~~\textit{where}~~} F=(f_1,f_2,\ldots ,f_m) ~\textit{and}~ z\in \T.$$ Let $C$ be the anti-linear involution (in fact a conjugation) in $\hdcm$ such that $CF(z)= \overline{F(\bar{z})}$, then for $U\in \bmoa(\T,\mathcal{L}(\C^m))$ we can similarly define the linear Hankel operator $G_U$ in $\hdcm$ by $$G_U(F) = P_m(UJ{F}).$$
	Then, we have $G_U = H_U C$ and $G_U^* = C H_U$, and therefore using our description\ref{main_structure_theorem}, one can have the precise structure of $\ker(G_U^*G_U -s^2I)$ and $\ker(G_U G_U^* -s^2I)$ in $\hdcm$. Note that, $$\ker(G_U^*G_U -s^2I)= C \smdt \hspace{0.5cm} \textit{and} \hspace{0.5cm} \ker(G_U G_U^* -s^2I) = \smdt .$$
\end{rmrk}

\section{The action of $H_U$ on $\smdt \equiv F_0\kcl_{\Theta}$}\label{action}
The \textit{Schmidt} subspaces $\smdt$ of the Hankel operator $H_U$ remain invariant under the operator $H_u$ in $\hdcm$. In \cite{GaPu20}\cite{GaPu21} the authors have discussed the explicit formula for the action of $H_u$ on these invariant $\smdts \equiv p\kcl_\theta$ in terms of the parameters $s,p$ and $\theta$ only. Also we have noted that in scalar valued Hardy space $\hdcc$ for any non-trivial \textit{Schmidt} subspace $\smdts$ it automatically holds that $\textit{dim}\left(\smdts \ominus (\smdts \cap z\hdcc)\right) =1$ when $\smdts \nsubseteq z\hdcc$. But in case of vector valued Hardy space $\hdcm$, when $\smdt \nsubseteq z\hdcm$, in general we have 
$$1\leq \text{dim}(\smdt \ominus (\smdt \cap z\hdcm)) \leq m $$ 
for any non-zero subspace $\smdt$. Under the assumption $\text{dim}(\smdt \ominus (\smdt$ $ \cap z\hdcm))=m$, in Theorem \ref{nearly_inv_thm}, we have seen that $~\smdt$ has compact form $\smdt = F_0\kcl_{\Theta}$ due to its \textit{nearly $S^*$-invariant} property. In this section we will obtain an explicit formula for the action of $H_u$ on $\smdt = F_0\kcl_{\Theta}$ under the assumption $\text{dim}(\smdt \ominus (\smdt \cap z\hdcm))=m$. Before going to the  main result of this section, we need some useful lemmas.

\begin{lma}\label{theta_a_sym}
	Let $\Theta \in H^\infty(\D,\mathcal{L}(\mathbb{C}^{m},\mathbb{C}^m))$ be an inner function with $\Theta (0)=0$ and $A\in M_m(\C)$ an unitary constant matrix such that $\Theta A$ is symmetric. Then  $S^*{\Theta} (A\bar{G})\in \kcl_\Theta$, for any $G \in \kcl_\Theta $.
\end{lma}
\begin{proof}
	First note that $\Theta A$ is also an inner function in $H^\infty(\D,\mathcal{L}(\mathbb{C}^{m},\mathbb{C}^m))$. Since, $\Theta A$ is symmetric we have $(\Theta A)^t = \Theta A$. For any $F\in \hdcm$, we get
	\begin{align*}
		\left\langle S^*\Theta (A\bar{G}), \overline{zF}\right\rangle_{L^2} & = \left\langle \Theta (A\bar{G}), \bar{F}\right\rangle_{L^2} \\
		& = \left\langle (\Theta A)^t F, G\right\rangle_{H^2} \\
		& = \left\langle \Theta A F, G\right\rangle_{H^2} =0 .
	\end{align*}
	Also, 
	\begin{align*}
		\left\langle S^*\Theta (A\bar{G}), \Theta H \right\rangle = \left\langle  A\bar{G}, z H \right\rangle = \left\langle  A\bar{G}, z H \right\rangle = \left\langle  \bar{G}, z A^*H \right\rangle = 0,
	\end{align*}
	for any $H\in \hdcm$. Hence, $S^*{\Theta} (A\bar{G})\in \kcl_\Theta$.
\end{proof}

\begin{lma}\label{k_theta_cm}
	For an inner function $\Theta \in H^\infty(\D,\mathcal{L}(\mathbb{C}^{r},\mathbb{C}^m))$,
	\begin{equation*}
		S^*\left(\kcl_{\Theta} \cap {\C^m}^{\perp} \right) = \kcl_{\Theta} \cap \left\{\bigvee_{i=1}^r S^*(\Theta e_i) \right\}^{\perp},
	\end{equation*}
	where $\left\{e_i\right\}_{i=1}^{r}$ is the standard orthonormal basis of $\C^r$.
\end{lma}
\begin{proof}
	Since $S$ is an isometry on $\hdcm$ then it is sufficient to proof
	\begin{equation*}
		\kcl_{\Theta} \cap {\C^m}^{\perp} = S\left [ \kcl_{\Theta} \cap \left\{\bigvee_{i=1}^r S^*(\Theta e_i) \right\}^{\perp} \right ].
	\end{equation*}
	Let $F\in \kcl_{\Theta} \cap {\C^m}^{\perp}$, then $F= SG$ for some $G\in \hdcm$. Therefore $G=S^*F$ and hence $G\in 	\kcl_{\Theta}$. Also for each $i\in \{1,2,\ldots,r\}$, 
	\begin{align*}
		\left \langle G, S^*(\Theta e_i)\right \rangle = \left \langle SG, \Theta e_i\right \rangle = \left \langle F, \Theta e_i\right \rangle =0,
	\end{align*}
	which implies that $G \in \left\{\bigvee_{i=1}^r S^*(\Theta e_i) \right\}^{\perp}$, and it further implies $$F \in S\left[\kcl_{\Theta} \cap \left\{\bigvee_{i=1}^r S^*(\Theta e_i) \right\}^{\perp} \right].$$

	Let  $F_1\in \hdcr$, then $F_1 = F_1(0)+SF_2$ for some $F_2 \in  \hdcr$. Again for any $G_1 \in \kcl_{\Theta} \cap \left\{\bigvee_{i=1}^r S^*(\Theta e_i) \right\}^{\perp}$ we have 
	\begin{align*}
		\left \langle SG_1,\Theta F_1\right \rangle & = 	\left \langle SG_1,\Theta (F_1(0)+SF_2) \right \rangle \\
		& = \left \langle SG_1,\Theta F_1(0)\right \rangle + \left \langle SG_1,S\Theta F_2 \right \rangle \\
		& = \left \langle SG_1,\Theta \left(\sum_{i=1}^{r}\langle F,e_i\rangle e_i\right)\right \rangle + \left \langle SG_1,S\Theta F_2 \right \rangle \\
		& = \sum_{i=1}^{r}\langle e_i, F\rangle \left \langle G_1,S^*\Theta e_i\right \rangle  + \left \langle G_1,\Theta F_2 \right \rangle =0,
	\end{align*}
	which proves that $SG_{1} \in \kcl_{\Theta} \cap {\C^m}^{\perp}$ and hence $$S^*\left(\kcl_{\Theta} \cap {\C^m}^{\perp} \right) = \kcl_{\Theta} \cap \left\{\bigvee_{i=1}^r S^*(\Theta e_i) \right\}^{\perp}.$$
\end{proof}
\begin{hypo}\label{dim}
	Let $U$ satisfies  Hypothesis \ref{usym}, $s>0$, and the \textit{Schmidt} subspace $\smdt$ of the Hankel operator $H_U$ satisfies the condition $\text{dim}(\smdt \ominus (\smdt \cap z\hdcm))=m$.
\end{hypo}
Now we are in a position to state the main theorem in this section regarding the action of the Hankel operator $H_U$ on $\smdt$.

\begin{thm}\label{action_thm}
	Let $U$ and $\smdt$ satisfy the Hypothesis \ref{dim}. Let $\{W_1,W_2,\ldots ,$ $W_m\} $ be an orthonormal basis of $~ \wcl:=\smdt \ominus (\smdt \cap z\hdcm)$. Then $\smdt$ has the following representation: $$\smdt = F_0\kcl_{\Theta},$$ where $F_0=[W_1,W_2,\ldots ,W_m]_{m\times m}$ and $\kcl_\Theta = \hdcm \ominus \Theta \hdcm$ for some inner function $\Theta \in H^\infty(\D,\mathcal{L}(\mathbb{C}^{m},\mathbb{C}^m))$ with $\Theta (0)=0$. Moreover, for any $G\in \kcl_\Theta$, there exists an unitary constant matrix $A\in M_m(\C)$ such that the action of $H_U$ on $\smdt$ is given by
	\begin{equation}\label{formula4.1}
		H_U(F_0G) = s\,\mathbb{P}_mF_0 \left[ S^*\Theta (A\bar{G}) \right], \hspace{0.5cm} G\in \kcl_{\Theta}.
	\end{equation}
\end{thm}
\begin{proof}
	Let  $W_i = [w_{1i},w_{2i},\ldots ,w_{mi}]^t $,  $i\in \{1,2,\ldots ,m \}$.
	Since, $\{e_1,e_2,\ldots ,e_m \}$ is the standard orthonormal basis of $\C^m$, then for each $j\in \{1,2,\ldots ,m \}$,
	\begin{align*}
		P_{\smdt}(e_j) & = \langle e_{j}, W_1\rangle W_1 + \langle e_{j}, W_2\rangle W_2+\cdots + \langle e_{j}, W_m\rangle W_m \\
		& = \overline{w_{j1}(0)} W_1 +\overline{w_{j2}(0)} W_2 +\cdots +\overline{w_{jm}(0)} W_m.
	\end{align*}
	Note that $H_U$ commutes with $H_U^2$, hence with the orthogonal projection onto $\smdt$. For $1 \leq i \leq m $, let $U_i^s$ denote the orthogonal projection of $U_i$ onto $\smdt$. Then we have,
	\begin{align}\label{cond4.1}
		\nonumber U_i^s = P_{\smdt}(U_i)  & = P_{\smdt}H_U(e_i)\\
		\nonumber & = H_U P_{\smdt}(e_i) \\
		\nonumber & = H_U[\overline{w_{i1}(0)} W_1 +\overline{w_{i2}(0)} W_2 +\cdots +\overline{w_{im}(0)} W_m] \\
		& = w_{i1}(0) H_U(W_1) +w_{i2}(0) H_U (W_2) +\cdots + w_{im}(0)H_U (W_m). 
	\end{align}
	Therefore we have
	\begin{equation}\label{equ4.1}
		[U_1^s , U_2^s, \ldots ,U_m^s]_{m\times m} = [H_U(W_1),H_U(W_2),\ldots ,H_U(W_m)]_{m\times m} [w_{ij}(0)]_{m\times m}^t.
	\end{equation}
	
	Now by Theorem \ref{nearly_inv_thm}, $~\smdt = F_0(\hdcm \ominus \Phi\hdcr)$ for some inner function $ \Theta \in H^\infty(\D,$ $\mathcal{L}(\mathbb{C}^{r},\mathbb{C}^m))$, with $\Theta (0) = 0 $ and $F_0=[W_1,W_2,\ldots ,W_m]_{m\times m}$.
	
	It is easy to check that
	\begin{enumerate}[(i)]
		\item $F_0\kcl_{\Theta} \cap {\C^m}^\perp = F_0\kcl_{\Theta} \cap {\wcl}^\perp $, and
		\item  $F_0\kcl_{\Theta} \cap \{U_1,U_2, \ldots ,U_m \}^\perp =	F_0\kcl_{\Theta} \cap \{U_1^s,U_2^s, \ldots ,U_m^s \}^\perp .$
	\end{enumerate}
	
	Due to the nearly $S^*$-invariant property of $\smdt = F_0\kcl_{\Theta}$ and using above two identity  along with \eqref{importantnew} we also have,
	\begin{equation}\label{equ4.2}
		S^*\left( F_0\kcl_{\Theta} \cap {\wcl}^\perp\right) \subseteq F_0\kcl_{\Theta} \cap \{U_1^s,U_2^s, \ldots ,U_m^s \}^\perp .
	\end{equation}
	By Lemma \ref{k_theta_cm},
	\begin{align}
		\nonumber S^*\left( F_0\kcl_{\Theta} \cap {\wcl}^\perp\right) & = S^*\left( F_0(\kcl_{\Theta} \cap {\C^m}^\perp)\right)\\
		\nonumber& = F_0S^*\left(\kcl_{\Theta} \cap {\C^m}^\perp\right)\\
		\nonumber& = F_0\left(\kcl_{\Theta} \cap \left\{\bigvee_{i=1}^r S^*(\Theta e_i) \right\}^{\perp}\right)\\
		& = F_0 \kcl_{\Theta} \cap \left\{\bigvee_{i=1}^r F_0S^*(\Theta e_i) \right\}^{\perp}.\label{equ4.3}
	\end{align}
	Combining \ref{equ4.2} and \ref{equ4.3}, we get 
	\begin{align}\label{equ4.4}
		F_0 \kcl_{\Theta} \cap \left\{\bigvee_{i=1}^r F_0S^*(\Theta e_i) \right\}^{\perp}\subseteq F_0\kcl_{\Theta} \cap \{U_1^s,U_2^s, \ldots ,U_m^s \}^\perp .
	\end{align}
	Next we show that $ \operatorname{span}\{U_1^s,U_2^s,\ldots ,U_m^s\} \subseteq \bigvee\limits_{i=1}^r \{F_0S^*(\Theta e_i)\}\subseteq F_0\kcl_\Theta $.  Let $F\in  \operatorname{span}\{U_1^s,U_2^s,\ldots ,U_m^s\}$ Therefore $F$ can be written as $$F=F_1\oplus F_2,$$ where $F_1 \in  \bigvee_{i=1}^r {F_0S^*(\Theta e_i)}$ and $F_2\in  F_0 \kcl_{\Theta} \cap \left\{\bigvee_{i=1}^r F_0S^*(\Theta e_i) \right\}^{\perp} $. So, we have
	$$\langle F,F_2\rangle =\langle F_1,F_2 \rangle + \langle F_2,F_2\rangle,$$
	which due to \eqref{equ4.4}, gives $\norm{F_2}^2=0$ that is, $F_2 =0$. In other words, $F=F_1 \in\bigvee_{i=1}^r {F_0S^*(\Theta e_i)} $ implies
	\begin{align}\label{equ4.5}
		\operatorname{span}\{U_1^s,U_2^s,\ldots ,U_m^s\} \subseteq \bigvee_{i=1}^r {F_0S^*(\Theta e_i)} .
	\end{align} 
	Hence, from the above set inclusion \eqref{equ4.5} we have
	\begin{equation}\label{equ4.7i}
		\begin{split}
			U_1^s &= c_{11}F_0S^*\Theta e_1 + c_{21}F_0S^*\Theta e_2 + \cdots + c_{r1}F_0S^*\Theta e_r\\
			U_{ 2}^s &= c_{12}F_0S^*\Theta e_1 + c_{22}F_0S^*\Theta e_2 + \cdots + c_{r2}F_0S^*\Theta e_r\\
			\vdots&  ~\hspace{4cm} \vdots \\
			U_{ m}^s &= c_{1m}F_0S^*\Theta e_1 + c_{2m}F_0S^*\Theta e_2 + \cdots + c_{rm}F_0S^*\Theta e_r,
		\end{split}
	\end{equation}
	where $c_{ij}, 1\leq i\leq r, 1\leq j\leq m$, are some scalars. In matrix notation,
	\begin{align}\label{new1}
		\nonumber[U_1^s,U_2^s,\ldots ,U_m^s]_{m\times m} & = [F_0S^*\Theta e_1, F_0S^*\Theta e_2, \ldots  ,F_0S^*\Theta e_r]_{m\times r}\,[c_{ij}]_{r\times m}\\
		&= F_0[S^*\Theta e_1, S^*\Theta e_2, \ldots  ,S^*\Theta e_r]_{m\times r}\,[c_{ij}]_{r\times m}.
	\end{align}
	From \eqref{equ4.1} and \eqref{new1}, it follows that
	\begin{align}\label{equ4.7}
		\nonumber&[H_U(W_1),H_U(W_2),\ldots ,H_U(W_m)]_{m\times m}  [w_{ij}(0)]_{m\times m}^{t}\\ & = F_0[S^*\Theta e_1, S^*\Theta e_2, \ldots  ,S^*\Theta e_r]_{m\times r}[c_{ij}]_{r\times m}.
	\end{align}
	In the proof of Theorem \ref{nearly_inv_thm}, we have seen that $\{W_1(0), W_2(0),\ldots ,W_m(0)\}$ is linearly independent in $\C^m$, so $[w_{ij}(0)]_{m\times m}$ is an invertible matrix and so is \\$ [w_{ij}(0)]_{m\times m}^t$. Therefore the equation \eqref{equ4.7} can be rewritten as,
	\begin{align}\label{equ4.9}
		\nonumber&[H_U(W_1),H_U(W_2),\ldots ,H_U(W_m)]\\
		 & = F_0[S^*\Theta e_1, S^*\Theta e_2, \ldots  ,S^*\Theta e_r][c_{ij}]\left([w_{ij}(0)]^t\right)^{-1}.
	\end{align}
	Now, for any $i,j \in \{1,2,\ldots ,m\}$, using \eqref{cond4.1} we have,
	\begin{align}\label{A1}
		\nonumber&\norm{U_i^s}^2\\
		\nonumber& = \left \langle U_i^s ,U_i^s\right \rangle \\
		\nonumber& = \langle w_{i1}(0) H_U(W_1) +\cdots + w_{im}(0)H_U (W_m),~ w_{i1}(0) H_U(W_1) +\\
		\nonumber&\hspace{2.5in}\cdots + w_{im}(0)H_U (W_m) \rangle \\
		\nonumber& = |{w_{i1}(0)}|^2 \langle  H_U^2(W_1),W_1  \rangle  +|{w_{i2}(0)}|^2 \langle  H_U^2(W_2),W_2  \rangle +\\
		\nonumber&\hspace{2in} \cdots + |{w_{im}(0)}|^2 \langle  H_U^2(W_m),W_m  \rangle \\
		& = s^2 \left\{|{w_{i1}(0)}|^2 +|{w_{i2}(0)}|^2+ \cdots +|{w_{im}(0)}|^2\right\}.
	\end{align}
	Similarly for $i\neq j$,
	\begin{align}\label{A2}
		\left \langle U_i^s ,U_j^s\right \rangle & = s^2 \left\{{w_{i1}(0)}\overline{w_{j1}(0)} +{w_{i2}(0)}\overline{w_{j2}(0)}+ \cdots +{w_{im}(0)}\overline{w_{jm}(0)}\right\}.	
	\end{align}
	Also from \eqref{equ4.7i} we can derive in a similar manner,
	\begin{align}\label{A3}
		\norm{U_i^s}^2 & = |c_{1i}|^2 + |c_{2i}|^2 + \cdots + |c_{ri}|^2
	\end{align}
	and for $i\neq j$,
	\begin{align}\label{A4}
		\left \langle U_i^s ,U_j^s\right \rangle & = c_{1i}\overline{c_{1j}} +	c_{2i}\overline{c_{2j}} + \cdots + c_{ri}\overline{c_{rj}}.
	\end{align}
	Combining \eqref{A1}, \eqref{A2}, \eqref{A3} and \eqref{A4}, we get
	\begin{align*}
		[c_{ij}]^*[c_{ij}] 
		=s^2\left([w_{ij}(0)]^t\right)^* [w_{ij}(0)]^t,
	\end{align*}
	which gives
	\begin{equation}\label{equ4.14}
		\Big(\left([w_{ij}(0)]^t\right)^{-1}\Big)^*\,[c_{ij}]^*\,[c_{ij}]\,\left([w_{ij}(0)]^t\right)^{-1}= s^2I.
	\end{equation}
	Let, $A=\frac{1}{s}[c_{ij}]\left([w_{ij}(0)]^t\right)^{-1}$, then by \eqref{equ4.14} we have $A^*A =I$, that is $A$ is an isometry from $\C^m$ to $\C^r$. Since by hypothesis, we have $r\leq m$ then we must have $r=m$. 
	
	So, $A$ is an unitary matrix from $\C^m$ to $\C^m$ and  \eqref{equ4.9} can be re written as 
	\begin{equation}
		[H_U(W_1),H_U(W_2),\ldots ,H_U(W_m)]  = s\,F_0\,[S^*\Theta e_1, S^*\Theta e_2, \ldots  ,S^*\Theta e_r]\,A.
	\end{equation}
	Next we derive the action of $H_U$ on dense subspace of  $F_0\kcl_\Theta$. So, consider any $G\equiv [g_1,g_2,\ldots ,g_m]^t\in \kcl_\Theta \cap H^\infty(\T,\mathcal{L}(\C,\C^m)$, then
	\begin{align*}
		H_U(F_0G) & = P_{m}\left(U\bar{F_0}\bar{G}\right) \\
		& = P_{m}\left\{U(\bar{g_1}\bar{W_1}+\bar{g_2}\bar{W_2} + \cdots +\bar{g_m}\bar{W_m})\right\} \\
		& = P_{m} \left\{ H_U(W_1)\bar{g_1} +H_U(W_2)\bar{g_2} +\cdots +H_U(W_m)\bar{g_m} \right\}\\
		& = P_{m} \left\{ [H_U(W_1),H_U(W_2),\cdots ,H_U(W_m)]\bar{G}\right\}\\
		& = sP_{m}\left\{ F_0[S^*\Theta e_1, S^*\Theta e_2, \cdots  ,S^*\Theta e_r]A\bar{G} \right\}.
	\end{align*}
	Therefore, by standard limiting argument, we get the explicit action of $H_U$ on $\smdt = F_0\kcl_\Theta$ as follows:
	$$	H_U(F_0G) = s\, \mathbb{P}_m\Big(F_0[S^*\Theta e_1, S^*\Theta e_2, \ldots  ,S^*\Theta e_r]A\bar{G}\Big), \hspace{0.5cm} G\in \kcl_{\Theta}.$$
\end{proof}
\begin{rmrk}
	In the above Theorem \ref{action_thm}, if $\Theta A$ is symmetric, then due to Lemma \ref{theta_a_sym}, the action  of $H_U$ (see \eqref{formula4.1}) on $\smdt$ is reduced to
	\begin{align}\label{new2}
		H_U(F_0G) = s\,F_0 \left[ S^*\Theta (A\bar{G}) \right], \hspace{0.5cm} G\in \kcl_{\Theta}.
	\end{align} If $m=1$, then $\Theta A$ is automatically symmetric. Therefore by \eqref{new2}, one can get back the action of the Hankel operator on the Schmidt subspace obtained in \cite{GaPu20,GaPu21}. 
\end{rmrk}

Next we see some useful observations concerning Theorem \ref{action_thm}.
\begin{rmrk}
	For the inner function $\Theta$ obtained in Theorem \ref{action_thm}, $S^*\Theta A $ is symmetric at $0$.
	\begin{proof}
		If $S^*\Theta A = [\theta_{ij}]_{m\times m}$, then our claim is to show $\theta_{ij}(0) = \theta_{ji}(0)$. Recall that $S^*\Theta Ae_i \in K_{\Theta}$. Now,
		\begin{align*}
			\theta_{ij}(0) &= \langle S^*\Theta A(e_i),e_j \rangle 
			= \langle F_{0} S^*\Theta A(e_i),F_{0}e_j \rangle 
			= \frac{1}{s}  \langle H_{U}(W_{i}) , W_{j} \rangle \\
			& = \frac{1}{s} \langle H_{U}(W_{j}) , W_{i} \rangle 
			= \langle F_{0} S^*\Theta A(e_j),F_{0}e_i \rangle 
			= \langle S^*\Theta A(e_j),e_i \rangle 
			= \theta_{ji}(0).
		\end{align*}
		Therefore we  conclude that $S^*\Theta A (0)$ is symmetric.	
	\end{proof}
\end{rmrk}

\begin{rmrk}
	It is surprising to note that, in Theorem \ref{action_thm},  the inner function $\Theta$ is {\bf unitary} almost everywhere on $\T$ which is not the case in general.
\end{rmrk}
As promised earlier, we conclude the section with an example of a class of non-trivial Schmidt subspaces satisfying the Hypothesis~ \ref{dim} in $H^2_{\C^2}(\D)$.
\begin{xmpl}\label{lastexample}
	Consider two non-constant inner functions $\phi, \psi \in \mathcal{H}^\infty (\D,\C)$ with $\phi \neq \psi$. For $\theta =\phi + \psi$ and $\gamma =\phi -\psi$, let
	\begin{equation*}
		U=
		\begin{bmatrix}  
			\theta & \gamma \\[1pt]
			\gamma & \theta
		\end{bmatrix} 
		\hspace{0.5cm} \textit{then,} \hspace{0.5cm}
		H_U =
		\begin{bmatrix}  
			H_{\theta} & H_{\gamma} \\[1pt]
			H_{\gamma} & H_{\theta}
		\end{bmatrix}. 
	\end{equation*}
	Then it is easy to see that, $$\smdt \ominus (\smdt \cap zH^2_{C^2}(\D)) = \bigvee \{\frac{s}{4}e_1,\frac{s}{4}e_2\},$$ 
	and hence $\text{dim}(\smdt \ominus (\smdt \cap zH^2_{C^2}(\D)))= 2$. The action of $H_U$ on $\smdt$ can be determined by the formula \eqref{formula4.1} given in Theorem \ref{action_thm}. It is remarkable to mention that, the scalar formula will not help to find the action of $H_U$ on $\smdt\subseteq\hdcm$.
\end{xmpl} 

\textit{Open question:} If $\text{dim}(\smdt \ominus (\smdt \cap z\hdcm))<m$, then it is still unknown to us what will be the formula for the action of $H_U$ on $\smdt$.

\section*{Acknowledgments}
		{\it The authors are extremely grateful to Prof. Alexander Pushnitski for giving his valuable comments on this work.  A. Chattopadhyay is supported by the Core Research Grant (CRG), File No: CRG/2023/004826, by the Science and Engineering Research Board (SERB), Department of Science \& Technology (DST), Government of India. S. Das acknowledges financial support from the Indian Statistical Institute Bangalore, India. C. Pradhan acknowledges support from JCB/2021/000041 as well as IoE post-doctoral fellowship from the Indian Institute of Science Bangalore, India. We sincerely thank the anonymous referee for several valuable suggestions.}

\end{document}